\documentclass[12pt, reqno]{amsart}
\usepackage{indentfirst, amssymb, amsmath, amsthm, mathrsfs, setspace, indentfirst, enumerate,  mathrsfs, amsmath, amsthm}
\usepackage{xcolor}
\usepackage[bookmarksnumbered, colorlinks, plainpages]{hyperref}
\usepackage{mathrsfs}
\usepackage{cite}

\newtheorem{ques}{Question}[section]

\newtheorem{theo}{Theorem}[section]
\newtheorem{lem}{Lemma}[section]

\newtheorem{conj}{Conjecture}[section]
\newtheorem{defi}{Definition}[section]

\newcommand{\ol}{\overline}

\newcommand{\be}{\begin{equation}}
\newcommand{\ee}{\end{equation}}
\newcommand{\beas}{\begin{eqnarray*}}
\newcommand{\eeas}{\end{eqnarray*}}
\newcommand{\bea}{\begin{eqnarray}}
\newcommand{\eea}{\end{eqnarray}}

\numberwithin{equation}{section}
\begin{document}

\title[S\MakeLowercase{harp Bounds and Related Problems for Functions in the Classes} $\mathcal{S}_e^{\ast}$ \MakeLowercase{and} $\mathcal{C}_e$]
{O\MakeLowercase{n} C\MakeLowercase{oefficient Problems for classes} $\mathcal{S}_e^{\ast}$ \MakeLowercase{and} $\mathcal{C}_e$.}

\date{}
\author[S. M\MakeLowercase{ajumder}, N. S\MakeLowercase{arkar}, \MakeLowercase{and} M. B. A\MakeLowercase{hamed}]{S\MakeLowercase{ujoy} M\MakeLowercase{ajumder}, N\MakeLowercase{abadwip} S\MakeLowercase{arkar}, M\MakeLowercase{olla} B\MakeLowercase{asir} A\MakeLowercase{hamed}$^*$}
\address{Department of Mathematics, Raiganj University, Raiganj, West Bengal-733134, India.}
\email{sm05math@gmail.com}

\address{Department of Mathematics, Raiganj University, Raiganj, West Bengal-733134, India.}
\email{naba.iitbmath@gmail.com}

\address{Molla Basir Ahamed,
	Department of Mathematics,
	Jadavpur University,
	Kolkata-700032, West Bengal, India.}
\email{mbahamed.math@jadavpuruniversity.in}

\renewcommand{\thefootnote}{}
\footnote{*\emph{Corresponding Author}: Molla Basir Ahamed.}

\renewcommand{\thefootnote}{\arabic{footnote}}
\setcounter{footnote}{0}

\begin{abstract}
Logarithmic coefficients play a crucial role in the theory of univalent functions. In this study,we focus on the classes $\mathcal{S}_e^\ast$ and $\mathcal{C}_e$ of starlike and convex functions, respectively, 

\begin{align*}
	\mathcal{S}_e^\ast := \left\{ f \in \mathcal{S} : \frac{zf'(z)}{f(z)} \prec e^z, \ z \in \mathbb{D} \right\},
\end{align*}
and
\begin{align*}
	\mathcal{C}_e := \left\{ f \in \mathcal{S} : 1 + \frac{z f''(z)}{f'(z)} \prec e^z, \ z \in \mathbb{D} \right\}.
\end{align*}
 This paper investigates the sharp bounds of the logarithmic coefficients and the Hermitian-Toeplitz determinant of these coefficients for the classes $\mathcal{S}_e^\ast$ and $\mathcal{C}_e$. Additionally, we examine the generalized Zalcman conjecture and the generalized Fekete-Szeg\"o inequality for these classes $\mathcal{S}_e^\ast$ and $\mathcal{C}_e$ and show that the inequalities are sharp.
\end{abstract}

\maketitle

\section{{\bf Introduction}}
Let $\mathcal{H}$ denote the class of analytic functions in the unit disk $\mathbb{D}:= \{ z \in \mathbb{C} : |z| < 1 \}$. Here $\mathcal{H}$ is a locally convex topological vector space endowed with the topology of uniform convergence over compact subsets of $\mathbb{D}$. Let  $\mathcal{A}$ denote the class of functions $f\in\mathcal{H}$ normalized by $f(0)=0=f'(0)-1$, and 
 $\mathcal{S}$ denote the class of functions $f\in \mathcal{A}$ which are univalent (i.e. one-to-one) in $\mathbb{D}$. Thus $f\in\mathcal{S}$ has the following representation
\begin{equation}\label{eq1}
f(z) = z + \sum_{n=2}^\infty a_n z^n.
\end{equation}
A function $f\in\mathcal{A}$ is called starlike (convex, receptively) if $f(\mathbb{D})$ is starlike with respect to the origin (convex, respectively). Denote by $\mathcal{S}^*$ and $\mathcal{S}$ the classes of starlike and convex functions in $\mathcal{S}$ respectively. It is well-known that a function $f\in\mathcal{A}$ belongs to $\mathcal{S}^*$ if, and only if, ${\rm Re}(zf'(z)/f(z))>0$ for $z\in\mathbb{D}$. Similarly, a function $f\in\mathcal{A}$ belongs to $\mathcal{C}$ if, and only if, ${\rm Re}(1+zf''(z)/f'(z))>0$ for $z\in\mathbb{D}$. from the above it is easy to see that $f\in\mathcal{C}$ if, and only if, $zf'\in\mathcal{S}^*$.\vspace{1.2mm}

Let $\mathbb{B}_0$ denote the class of analytic functions $\omega$ in $\mathbb{D}$ with $\omega(0)=0$ and $|\omega(z)|<1$ for all $z \in \mathbb{D}$. Functions in $\mathbb{B}_0$ are known as \emph{Schwarz functions}.  
A function $\omega \in \Omega$ can be expressed as a power series $\omega(z) = \sum_{n=1}^\infty \omega_n z^n\; \mbox{for}\;z \in \mathbb{D}.$\vspace{1.2mm}

We now recall an important concept: subordination, which is a useful tool for solving challenging problems in geometric function theory.
\begin{defi}
	For two analytic functions $f$ and $g$ in a domain $\mathbb{D}$, we say that $f$ is \emph{subordinate} to $g$ in $\mathbb{D}$, and write $f \prec g$, if there exists a Schwarz function $\omega \in \Omega$ such that $f(z) = g(\omega(z)),\;z \in \mathbb{D}.$
	In particular, if $g$ is univalent in $\mathbb{D}$, then $f \prec g$ if and only if $f(0)=g(0)$ and $f(\mathbb{D}) \subset g(\mathbb{D})$.
\end{defi}
Using the subordination principle, Ma and Minda \cite{MM} introduced a unified framework for various subclasses of starlike functions in $1992$. They defined
\begin{align*}
	\mathcal{S}^*(\psi) := \left\{ f \in \mathcal{S} : \frac{zf'(z)}{f(z)} \prec \psi(z), \ z \in \mathbb{D} \right\},
\end{align*}
and
\begin{align*}
	\mathcal{C}(\psi) := \left\{ f \in \mathcal{S} : 1 + \frac{z f''(z)}{f'(z)} \prec \psi(z), \ z \in \mathbb{D} \right\},
\end{align*}
where $\psi$ is an analytic univalent function with positive real part in $\mathbb{D}$, symmetric with respect to the real axis, $\psi(0)=1$, and $\psi'(0)>0$.

Interest has grown in studying subclasses of starlike and convex functions for which the superordinate function $\psi(z)$ does not map the entire right half-plane. Although the exponential function is a natural choice for the superordinate function, its selection presents interesting and often non-trivial challenges. 

The class of starlike functions related to the exponential function $e^z$, $\mathcal{S}_e^\ast$, was introduced by Mendiratta \cite{Mendiratta2015} and is defined by the condition $\frac{zf'(z)}{f(z)} \prec e^z$. We also recall the related class $\mathcal{C}_e$ og convex functions related to the exponential function, defined by $1 + \frac{z f''(z)}{f'(z)} \prec e^z$. Precisely, the classes  $\mathcal{S}_e^\ast$ and $\mathcal{C}_e$ are defined as
\begin{align*}
	\mathcal{S}_e^\ast &:= \left\{ f \in \mathcal{S} : \frac{zf'(z)}{f(z)} \prec e^z, \ z \in \mathbb{D} \right\},
\end{align*}\begin{align*}
	\;\;\;\;\;\;\mathcal{C}_e:= \left\{ f \in \mathcal{S} : 1 + \frac{z f''(z)}{f'(z)} \prec e^z, \ z \in \mathbb{D} \right\}.
\end{align*}
\subsection{Logarithmic coefficients}
Note that for $f\in\mathcal{S}$, let  
\begin{align}
	\label{e3} F_f(z):= \log\frac{f(z)}{z} = 2 \sum_{n=1}^\infty \gamma_n z^n, \quad z \in \mathbb{D},\quad \log 1:=0.
\end{align}
The numbers $\gamma_n := \gamma_n(f)$ are the logarithmic coefficients of $f$. Few exact upper bounds for $\gamma_n$ exist. These coefficients are known to play a crucial role in the Miliin conjecture (\cite{IMM1}, see also \cite[p. 155]{PLD1}). Specifically, Miliin \cite{IMM1} conjectured that for $f\in\mathcal{S}$ and $n\geq 2$, 
\begin{align*}
	\sum\limits_{m=1}^n \sum\limits_{k=1}^m \left(k|\gamma_k|^2 - \frac{1}{k}\right) \leq 0.
\end{align*}
This conjecture was established by De Branges \cite{LDB1} in his proof of the Bieberbach conjecture. For the Koebe function $k(z) = \frac{z}{(1-z)^2}$, the logarithmic coefficients are given by $\gamma_n = \frac{1}{n}$. Since the Koebe function is the extremal function for many extremal problems in $\mathcal{S}$, it is natural to conjecture that $|\gamma_n| \leq \frac{1}{n}$ for all $f \in \mathcal{S}$. However, this conjecture does not hold universally. For example, there exists a bounded function $f \in \mathcal{S}$ with logarithmic coefficients $\gamma_n \neq O(n^{-0.83})$ (see \cite[Theorem 8.4]{PLD1}).\vspace{1.2mm}
 
By differentiating \eqref{e3} and comparing coefficients, the following expressions for $\gamma_n$ in terms of $a_n$ are obtained:  

\begin{align}\label{Eq-1.3}
	\begin{cases}
		\gamma_1 & = \dfrac{1}{2} a_2,\vspace{2mm} \\
		\gamma_2 & = \dfrac{1}{2}\left(a_3 - \dfrac{1}{2}a_2^2\right),\vspace{2mm} \\
		\gamma_3 & = \dfrac{1}{2}\left(a_4 - a_2a_3 + \dfrac{1}{3}a_2^3\right)\vspace{2mm}\\
		\gamma_4 &=\dfrac{1}{2}\big(a_5 - a_2 a_4 + a_2^2 a_3 - \dfrac{1}{2} a_3^2 - \dfrac{1}{4} a_2^4\big).
	\end{cases}
\end{align}
If $f \in \mathcal{S}$, it is straightforward to show that $|\gamma_1| \leq 1$, since $|a_2| \leq 2$. Using the Fekete-Szeg\"o inequality (see \cite[Theorem 3.8]{PLD1}) for functions in $\mathcal{S}$ and substituting into (\ref{e3}), the sharp estimate for $\gamma_2$ is given by 
\begin{align*}
	\left|\gamma_2\right| \leq \frac{1}{2}(1 + 2e^{-2}) \approx 0.635.
\end{align*}
For $n \geq 3$, deriving bounds for $|\gamma_n|$ is considerably more challenging, and no significant general bounds for $|\gamma_n|$ for functions in $\mathcal{S}$ are currently known. Logarithmic coefficients have recently been a focus of research interest for various authors (e.g., \cite{AA1,AA2,CKKLS2,PS0,PSW1,RK1,DT1}).\vspace{1.2mm}

In this article, we investigate various coefficient problems and determine their sharp bounds for several topics in geometric function theory, specifically focusing on the logarithmic coefficients, Hermitian-Toeplitz determinant, generalized Zalcman conjecture, and the generalized Fekete-Szeg\"o inequality. The remainder of the paper is organized as follows: Section \ref{Sec-2} introduces the necessary lemmas required to establish our main findings. Section \ref{Sec-3} establishes sharp bounds for the logarithmic coefficients of the classes $\mathcal{S}_e^{\ast}$ and $\mathcal{C}_e$. Section \ref{Sec-4} presents sharp bounds of the Second-order Hermitian-Toeplitz determinant of logarithmic coefficients for the classes $\mathcal{S}_e^{\ast}$ and $\mathcal{C}_e$. In Section \ref{Sec-5}, the generalized Zalcman conjecture for the classes $\mathcal{S}_e^{\ast}$ and $\mathcal{C}_e$ is discussed. Finally, Section \ref{Sec-6} establishes sharp bounds of the generalized Fekete-Szeg\"o functional for the classes $\mathcal{S}_e^{\ast}$ and $\mathcal{C}_e$. The proofs of the main results are discussed in detail in each respective section.

\section{{\bf Auxulary Lemmas}}\label{Sec-2}
Let $\mathcal{P}$ be the class of all analytic functions $p$ in the unit disk $\mathbb{D}$ such that $p(0) = 1$ and $\operatorname{Re} p(z) > 0$ for all $z \in \mathbb{D}$. Every $p \in \mathcal{P}$ then has the series representation
\begin{equation}\label{p1}
	p(z) = 1 + \sum_{n=1}^{\infty} c_n z^n, \quad z \in \mathbb{D}.
\end{equation}

Functions in $\mathcal{P}$ are referred to as Carath$\acute{e}$odory functions. It is well-known that for $p \in \mathcal{P}$, the coefficients satisfy the sharp bound $|c_n| \leq 2$ for all $n \geq 1$ (see \cite{PLD1}). The CarathM$\acute{e}$odory class $\mathcal{P}$ and its coefficient bounds play a fundamental role in deriving sharp estimates in geometric function theory.\vspace{2mm}

Now we recall the following well-known results due to Cho et al. \cite{CKL1}, which will play a key role in establishing the main results of this paper.
\begin{lem}\label{L1} \cite[Lemma 2.4]{CKL1} If $p\in\mathcal{P}$ is of the form (\ref{p1}), then
\bea
\label{L1}c_1 =2\tau_1,\eea
\bea\label{L2} c_2=2\tau_1^2 + 2(1 - \tau_1^2)\tau_2\eea
and
\bea
\label{L3} c_3 = 2\tau_1^3+4(1-\tau_1^2)\tau_1\tau_2 - 2(1 - \tau_1^2)\tau_1\tau_2^2 + 2(1 - \tau_1^2)(1 - |\tau_2|^2)\tau_3
\eea
for some $\tau_1, \tau_2, \tau_3 \in\mathbb{\ol D}:= \{z \in \mathbb{C}: |z| \leq 1 \}$.

\medskip
For $\tau_1 \in \mathbb{T}:= \{z \in \mathbb{C}: |z| = 1 \}$, there is a unique function $p\in \mathcal{P}$ with $c_1$  as in (\ref{L1}), namely
\[p(z)=\frac{1+\tau_1 z}{1 - \tau_1 z}, \quad z \in \mathbb{D}.\]

\medskip
For $\tau_1 \in \mathbb{D}$ and $\tau_2 \in \mathbb{T}$, there is a unique function $p\in \mathcal{P}$ with $c_1$ and $c_2$ as in (\ref{L1}) and (\ref{L2}), namely
\[p(z) = \frac{1+(\ol \tau_1 \tau_2+\tau_1)z + \tau_2 z^2}{1+(\ol \tau_1 \tau_2 - \tau_1)z - \tau_2 z^2}, \quad z \in \mathbb{D}.\]

\medskip
For $\tau_1, \tau_2 \in \mathbb{D}$ and $\tau_3 \in \mathbb{T}$, there is a unique function $p\in\mathcal{P}$ with $c_1$, $c_2$ and $c_3$  as in (\ref{L1})-(\ref{L3}), namely
\[p(z)=\frac{1 + (\ol\tau_2 \tau_3 + \ol\tau_1 \tau_2 + \tau_1)z+(\ol\tau_1\tau_3+\tau_1\ol\tau_2\tau_3+\tau_2)z^2+\tau_3z^3}{1+(\ol\tau_2\tau_3+\ol\tau_1\tau_2-\tau_1)z+(\ol\tau_1\tau_3-\tau_1\ol\tau_2\tau_3-\tau_2)z^2-\tau_3z^3},\;\;z\in\mathbb{D}.\]
\end{lem}
\begin{lem}\label{L2}\cite{CKS1} Let $A$, $B$, $C$  be real numbers and let
\[Y(A, B, C):= \max\limits_{z\in \ol{\mathbb{D}}}\left\lbrace |A+Bz+Cz^2|+1-|z|^2\right\rbrace.\]

\begin{enumerate} 
\item[(i)] If $AC\geq 0$, then
\begin{align*}
	Y(A, B, C) =
	\begin{cases}
		|A|+|B|+|C|, & \text{if}\;\;\; |B|\geq 2(1-|C|), \\
		1+|A|+\frac{B^2}{4(1-|C|)}, &\text{if}\;\;\; |B|<2(1-|C|).
	\end{cases}
\end{align*}
\item[(ii)] If $AC<0$, then 
\begin{align*}
	Y(A,B,C)=
	\begin{cases}
		1-|A|+\frac{B^2}{4(1-|C|)}, &\text{if}\;\;\; -4AC(C^{-2}-1) \leq B^2\; \text{and}\; |B|<2(1-|C|), \\
		1+|A|+\frac{B^2}{4(1+|C|)}, &\text{if}\;\;\; B^2<\min\left\{4(1+|C|)^2, -4AC(C^{-2}-1) \right\}, \\
		R(A,B,C), &\text{otherwise},
	\end{cases}
\end{align*}
where
\[R(A,B,C):=
\begin{cases}
|A|+|B|-|C|, & \text{if}\;\;\; |C|(|B|+4|A|) \leq |AB|, \\
-|A|+|B|+|C|, & \text{if}\;\;\; |AB|\leq |C|(|B|-4|A|), \\
(|C|+|A| )\sqrt{1-\frac{B^2}{4AC}}, &\text{otherwise}.
\end{cases}
\]
\end{enumerate} 
\end{lem}
\begin{lem}\label{L3} \cite{MM}
Let $p \in \mathcal{P}$ be given by \eqref{p1}. Then
\[
\left| c_2 - v c_1^2 \right| \le 
\begin{cases}
-4v + 2, & v < 0, \\
2, & 0 \leq v \leq 1, \\
4v - 2, & v > 1.
\end{cases}
\]

Moreover, for $v < 0$ or $v > 1$, equality holds if and only if
\[
h(z) = \frac{1+z}{1-z} \quad \text{or one of its rotations}.
\]

For $0 < v < 1$, equality holds if and only if
\[
h(z) = \frac{1+z^2}{1-z^2} \quad \text{or one of its rotations}.
\]
\end{lem}

\begin{lem}\label{L4}\cite{Ali2003}
Let $p \in \mathcal{P}$ be given by \eqref{p1} with $0 \leq B \leq 1$ and $B(2B - 1) \leq D \leq B$. Then
\[
\left| c_3 - 2 B c_1 c_2 + D c_1^3 \right| \leq 2.
\]
\end{lem}

\begin{lem}\label{L5}\cite{Ravichandran2015}
Let $p \in \mathcal{P}$ be given by \eqref{p1}. If $\alpha, \beta, \gamma, \lambda$ satisfy
\[
0 < \alpha < 1, \quad 0 < \lambda < 1,
\]
and
\begin{align*}
	8&\lambda(1 - \lambda) \Big\{ (\alpha\beta - 2\gamma)^2 + (\alpha(\lambda + \alpha) - \beta)^2 \Big\}
	+ \alpha(1 - \alpha)(\beta - 2\lambda \alpha)^2\\& \leq 4 \alpha^2 (1 - \alpha)^2 \lambda (1 - \lambda),
\end{align*}
then
\[
|\gamma c_1^4 + \lambda c_2^2 + 2 \alpha c_1 c_3 - \frac{3}{2} \beta c_1^2 c_2 - c_4 |\leq 2.
\]
\end{lem}

\begin{lem}\label{Ls}\cite{SimThomas2020}
Let $J, K,$ and $L$ be numbers such that $J \geq 0$, $K \in \mathbb{C}$, and $L \in \mathbb{R}$. 
Let $p \in \mathcal{P}$ be of the form (\ref{p1}) and define a function by

\[
\Phi(c_1,c_2) = \big| K c_1^2 + L c_2 \big| - \big| J c_1 \big|.
\]
Then 
\[
\Phi(c_{1}, c_{2}) \le 
\begin{cases}
|4K + 2L| - 2J, & \text{if } |2K + L| \geq |L| + J, \\[6pt]
2|L|, & \text{otherwise.}
\end{cases}
\]
and
\[
-\Phi(c_1,c_2) \leq 
\begin{cases}
2J - M, & \text{when } J \geq M + 2|L|, \\[6pt]
2J \sqrt{\dfrac{ \cdot 2|L|}{M + 2|L|}}, & \text{when } J^2 \leq 2|L|(M + 2|L|), \\[10pt]
2|L| +\dfrac{ J^2}{M + 2|L|}, & \text{otherwise},
\end{cases}
\]
where $M=|4K+2L|$.
\end{lem}
\section{{\bf Sharp Bounds for logarithmic coefficients for the classes $\mathcal{S}_e^{\ast}$ and $\mathcal{C}_e$.}}\label{Sec-3}
The central role of logarithmic coefficients in geometric function theory motivates efforts to obtain sharp estimates for them. In this section, we establish the following sharp bound for the logarithmic coefficients of functions in the classes $\mathcal{S}_e^{\ast}$ and $\mathcal{C}_e$.
\begin{theo} Let $f(z) = z + a_2 z^2 + a_3 z^3 + \cdots \in \mathcal{S}_e^{\ast}$ and $\gamma_1, \gamma_2, \gamma_3, \gamma_4$ be given by \eqref{Eq-1.3}. Then we have
\begin{align*}
	|\gamma_n| \le \frac{1}{2n}, \quad \text{for } n = 1, 2, 3, 4.
\end{align*}
All these bounds are sharp.
\end{theo}
The following conjecture is proposed for the general coefficients $\gamma_n$ ($n\geq 5$) of functions in the class $\mathcal{S}_e^{\ast}$.
\begin{conj}
	If $f(z) = z + a_2 z^2 + a_3 z^3 + \cdots \in \mathcal{S}_e^{\ast}$, then 
	\begin{align*}
		|\gamma_n| \le \frac{1}{2n} \quad \text{for } n\in\mathbb{N}.
	\end{align*}
	The bound is sharp for the functions $f_n$ (for each $n\in\mathbb{N}$) defined by $(\ref{te1})$ with
	\begin{align*}
		p(z) = \frac{1+z^n}{1-z^n}.
	\end{align*}
\end{conj}

\begin{proof} Let $f \in \mathcal{S}_e^{\ast}$. Then there exists a Schwarz function $w$ with $w(0) = 0$ and $|w(z)| < |z|$ for $z \in \mathbb{D}$ such that
\begin{equation}\label{te1}
\frac{z f'(z)}{f(z)} = e^{\,w(z)}.
\end{equation}

Let $p \in \mathcal{P}$. By applying the definition of subordination, we can express $p$ as
\begin{equation}\label{te2}
w(z) = \frac{p(z) - 1}{p(z) + 1}.
\end{equation}

Assuming $p$ is given by (\ref{p1}), equating coefficients from \eqref{eq1}, \eqref{te1}, and \eqref{te2} yields
\begin{align}
a_2 &= \frac{1}{2} c_1, \label{a12}\\
a_3 &= \frac{1}{16} c_1^2 + \frac{1}{4} c_2, \label{a13}\\
a_4 &= \frac{1}{24} c_1 c_2 - \frac{1}{288} c_1^3 + \frac{1}{6} c_3, \label{a14}\\
a_5 &= \frac{1}{1152} c_1^4 - \frac{1}{96} c_2c_1^2 + \frac{1}{48} c_1 c_3 + \frac{1}{8} c_4. \label{a15}
\end{align}

\noindent{\bf (A):} {\bf Sharp bounds of $\gamma_1$}: Using (\ref{a12}) and \eqref{Eq-1.3}, we have
\begin{align*}
	|\gamma_1| = \frac{1}{2} |a_2| = \frac{1}{4} |c_1| \le \frac{1}{2}.
\end{align*}
Hence, the desired bound is established. To establish the sharpness of the inequality, let us consider the function $f_1$ defined by (\ref{te1}) with
\begin{align*}
	p(z) = \frac{1+z}{1-z}.
\end{align*}
In this case, $f_{1} \in \mathcal{S}_e^{\ast}$, and its expansion is
\bea\label{f1}
f_{1}(z) = z+z^2+\frac{3}{4}z^3+\cdots 
\eea
and we see that 
\begin{align*}
	|\gamma_1|=\frac{1}{2}|a_2|=\frac{1}{2}.
\end{align*}

\noindent{\bf (B):} {\bf Sharp bounds of $\gamma_2$}: From (\ref{a12}), (\ref{a13}) and \eqref{Eq-1.3}, we obtain
\begin{align*}
	|\gamma_2| &= \bigg|\frac{1}{2} \left(a_3-\frac{1}{2}a_2^2\right)\bigg|\\
	&=\frac{1}{2} \bigg| \left(\frac{1}{16} c_1^2 + \frac{1}{4} c_2\right) - \frac{c_1^2}{8}\bigg|\\
	&=\frac{1}{8}|c_2-\frac{1}{4}c_1^2|.
\end{align*}
Thus, by Lemma \ref{L3}, we obtain the desired inequality
\begin{align*}
	|\gamma_2|\leq \frac{1}{4}.
\end{align*}
To establish the sharpness of the inequality, let us consider the function $f_{2}$ 
defined by (\ref{te1}) with 
\begin{align*}
	p(z) = \frac{1+z^2}{1-z^2}.
\end{align*}
In this case, we have $f_{2} \in  \mathcal{S}_e^{\ast}$ and its expansion is given by
\bea\label{f2}
f_{2}(z) = z+\frac{1}{2}z^3+\cdots 
\eea
and we see that 
\begin{align*}
	|\gamma_2|=\bigg|\frac{1}{2} \left(a_3-\frac{1}{2}a_2^2\right)\bigg|= \frac{1}{4}.
\end{align*}
\noindent{\bf (C):} {\bf Sharp bounds of $\gamma_3$}: Using (\ref{a12})-(\ref{a14}) and \eqref{Eq-1.3}, it follows that
\begin{align*}
|\gamma_3| &= \frac{1}{2} \left| a_4 - a_2 a_3 + \frac{1}{3} a_2^3 \right| \\
&= \frac{1}{2} \Bigg| \left(\frac{1}{24} c_1 c_2 - \frac{1}{288} c_1^3 + \frac{1}{6} c_3 \right)
- \left(\frac{1}{2} c_1 \right) \left(\frac{1}{16} c_1^2 + \frac{1}{4} c_2\right)
+ \frac{1}{3} \left(\frac{1}{2} c_1 \right)^3 \Bigg| \\
&= \frac{1}{12} \left| \frac{1}{24} c_1^3 - \frac{1}{2} c_1 c_2 + c_3 \right| \\
&= \frac{1}{12} \left| c_3 - 2 B c_1 c_2 + D c_1^3 \right|,
\end{align*}
where $B = \frac{1}{4}$ and $D = \frac{1}{24}$. Clearly, we have $0 < B < 1$, $D < B$, and $B(2B-1) = -\frac{1}{8} < D$. Hence, all the conditions of Lemma \ref{L4} are satisfied, and we obtain
\begin{align*}
	|c_3 - 2 B c_1 c_2 + D c_1^3| \le 2.
\end{align*}
Thus, we obtain the desired bound
\begin{align*}
	|\gamma_3| \le \frac{1}{6}.
\end{align*}
To establish the sharpness of the inequality, we consider the function $f_{3}$ 
defined in (\ref{te1}) with 
\begin{align*}
	p(z) = \frac{1+z^3}{1-z^3}.
\end{align*}
Clearly, $f_{3}\in\mathcal{S}_e^{\ast}$, and its series expansion is given by
\begin{align*}
	f_{3}(z) = z + \frac{1}{3} z^4 + \cdots.
\end{align*}
We see that 
\begin{align*}
	|\gamma_3| = \frac{1}{2} \left| a_4 - a_2 a_3 + \frac{1}{3} a_2^3 \right|=\frac{1}{6}.
\end{align*}

\noindent{\bf (D):} {\bf Sharp bounds of $\gamma_4$}: Using (\ref{a12})-(\ref{a15}) and \eqref{Eq-1.3}, it follows that
\begin{align*}
|\gamma_4| &= \frac{1}{2}\bigg|a_5 - a_2 a_4 + a_2^2 a_3 - \frac{1}{2} a_3^2 - \frac{1}{4} a_2^4\bigg| \\
&= \frac{1}{2}\bigg|\left( \frac{1}{1152} c_1^4 - \frac{1}{96} c_2 c_1^2 + \frac{1}{48} c_1 c_3 + \frac{1}{8} c_4 \right)
- \left( \frac{1}{48} c_1^2 c_2 - \frac{1}{576} c_1^4 + \frac{1}{12} c_1 c_3 \right) \\
&\quad + \left( \frac{1}{64} c_1^4 + \frac{1}{16} c_1^2 c_2 \right)
- \left( \frac{1}{512} c_1^4 + \frac{1}{64} c_1^2 c_2 + \frac{1}{32} c_2^2 \right)
- \frac{1}{64} c_1^4\bigg|\\
&=\frac{1}{16} \left| -\frac{1}{192} c_1^4 - \frac{1}{8} c_1^2 c_2 + \frac{1}{2} c_1 c_3 + \frac{1}{4} c_2^2 - c_4 \right|\\
&=\frac{1}{16} \left| \gamma  c_1^4 - \frac{3}{2} \beta c_1^2 c_2 +2\alpha c_1 c_3 + \lambda  c_2^2 - c_4 \right|
\end{align*}
where $\gamma =-\frac{1}{192}$, $\lambda  = \frac{1}{4}$, $\beta=\frac{1}{12}$ and $\alpha = \frac{1}{4}$. Observe that
\begin{align*}
	\begin{aligned}
		& 8 \lambda (1 - \lambda)\Big\{ (\alpha \beta - 2 \gamma)^2 + (\alpha (\lambda + \alpha) - \beta)^2 \Big\} 
		+ \alpha (1 - \alpha) (\beta - 2 \lambda \alpha)^2 \\
		&\quad- 4 \alpha^2 (1 - \alpha)^2 \lambda (1 - \lambda) = \frac{45}{2048} < 0.
	\end{aligned}
\end{align*}

This confirms that all hypotheses of Lemma \ref{L5} are satisfied. Consequently, we have
\begin{align*}
	\left| \gamma  c_1^4 - \frac{3}{2} \beta c_1^2 c_2 + 2\alpha c_1 c_3 + \lambda c_2^2 - c_4 \right| \le 2,
\end{align*}
which immediately implies the desired inequality
\begin{align*}
	|\gamma_4| \le \frac{1}{8}.
\end{align*}

To establish the sharpness of this bound, we consider the function $f_{4}$ defined in \eqref{te1}, where
\begin{align*}
	p(z) = \frac{1+z^4}{1-z^4}.
\end{align*}

We see that \(f_{4}\in\mathcal{S}_e^{\ast}\), and its series expansion is given by
\begin{align*}
	f_{4}(z) = z + \frac{1}{4} z^5 + \cdots.
\end{align*}
It is easy to see that 
\begin{align*}
	|\gamma_4| = \frac{1}{2}\bigg|a_5 - a_2 a_4 + a_2^2 a_3 - \frac{1}{2} a_3^2 - \frac{1}{4} a_2^4\bigg|=\frac{1}{8}.
\end{align*}
This completes the proof.
\end{proof}
\begin{theo} Let $f(z) = z + a_2 z^2 + a_3 z^3 + \cdots \in \mathcal{C}_e$ and $\gamma_1, \gamma_2, \gamma_3, \gamma_4$ be given by \eqref{Eq-1.3}. Then we have
\begin{align*}
	|\gamma_n| \leq
	\begin{cases}
		 \dfrac{1}{2n(n+1)},\;\; n=1, 2, 3,\vspace{2mm}\\
		 \dfrac{1}{8}, \;\;\;\;\;\;\;\;\;\;\;\;\;\;\;n=4.
	\end{cases}
\end{align*}
All these bounds are sharp.
\end{theo}
\begin{proof} Let \(f \in \mathcal{C}_e\). By the definition of subordination, we have
\begin{equation}\label{e1}
1 + \frac{z f''(z)}{f'(z)} = e^{w(z)},
\end{equation}
where \(w\) is analytic in \(\mathbb{D}\) with \(w(0) = 0\) and \(|w(z)| < |z|\) for all \(z \in \mathbb{D}\).  

Let \(p\) be defined as in \eqref{p1}. Combining \eqref{e1} with \eqref{te2}, we obtain the following relations for the coefficients of \(f\):
\begin{align}
a_2 &= \frac{1}{4} c_1, \label{a22} \\
a_3 &= \frac{1}{12} c_2 + \frac{1}{48} c_1^2, \label{a23} \\
a_4 &= \frac{1}{96} c_1 c_2 - \frac{1}{1152} c_1^3 + \frac{1}{24} c_3, \label{a24} \\
a_5 &= \frac{1}{5760} c_1^4 - \frac{1}{480} c_1^2 c_2 + \frac{1}{240} c_1 c_3 + \frac{1}{40} c_4. \label{a25}
\end{align}

\item[(A)] {\bf Sharp bounds of $\gamma_1$}: By applying (\ref{a22}) and \eqref{Eq-1.3}, we obtain
\begin{align*}
	|\gamma_1| = \frac{1}{2} |a_2| = \frac{1}{8} |c_1| \le \frac{1}{4}.
\end{align*}
Thus, the required bound is proved.\vspace{1.2mm}  

To show the sharpness of this inequality, we consider the function \(f_{5}\) defined by (\ref{te1}) with
\[
p(z) = \frac{1+z}{1-z}.
\]
For this particular choice, the function \(f_{5}\) belongs to the class \(\mathcal{C}_e\), and its series expansion is given by
\begin{equation}\label{f5}
f_{5}(z) = z + \frac{1}{2}z^2 + \frac{1}{4}z^3 + \cdots.
\end{equation}
We see that 
\begin{align*}
	|\gamma_1|=\frac{1}{2}|a_2|=\frac{1}{4}.
\end{align*}

\item[(B)] {\bf Sharp bounds of $\gamma_2$}: Using (\ref{a22}), (\ref{a23}) and \eqref{Eq-1.3}, it follows that
\beas
|\gamma_2| &=& \bigg|\frac{1}{2} \left(a_3-\frac{1}{2}a_2^2\right)\bigg|\\
&=&\frac{1}{2} \bigg| \left(\frac{1}{48} c_1^2 + \frac{1}{12} c_2\right) - \frac{c_1^2}{32}\bigg|\\
&=&\frac{1}{24}|c_2-\frac{1}{8}c_1^2|.\eeas
Thus, by Lemma \ref{L3}, we obtain the desired inequality
\[  |\gamma_2|\leq \frac{1}{12}.\]
To establish the sharpness of the inequality, let us consider the function $f_{6}$ 
defined by (\ref{te1}) with 
\[
p(z) = \frac{1+z^2}{1-z^2}.
\]
In this case, we have $f_{6} \in \mathcal{C}_e$ and its expansion is given by
\bea\label{f6}
f_{6}(z) = z+\frac{1}{6}z^3+\cdots .
\eea
We see that 
\begin{align*}
	|\gamma_2|=\bigg|\frac{1}{2} \left(a_3-\frac{1}{2}a_2^2\right)\bigg|= \frac{1}{12}.
\end{align*}
\item[(C)]{\bf Sharp bounds of $\gamma_3$}: Using (\ref{a22})-(\ref{a24}) and \eqref{Eq-1.3}, it follows that
\begin{align*}
|\gamma_3| &= \frac{1}{2} \left| a_4 - a_2 a_3 + \frac{1}{3} a_2^3 \right| \\
&= \frac{1}{2} \Bigg| \left( \frac{1}{96} c_1 c_2 - \frac{1}{1152} c_1^3 + \frac{1}{24} c_3 \right)
- \left(\frac{1}{4} c_1 \right) \left(\frac{1}{12} c_2 + \frac{1}{48} c_1^2\right)
+ \frac{1}{3} \left(\frac{1}{4} c_1 \right)^3 \Bigg| \\
&= \frac{1}{48} \left|  c_3 - \frac{1}{4} c_1 c_2 - \frac{1}{48} c_1^3 \right| \\
&= \frac{1}{48} \left| c_3 - 2 B c_1 c_2 + D c_1^3 \right|,
\end{align*}
where $B = \frac{1}{8}$ and $D = - \frac{1}{48}$. Clearly, we have $0 < B < 1$, $D < B$, and $B(2B-1) = -\frac{3}{32} < D$. Hence, all the conditions of Lemma \ref{L3} are satisfied, and we obtain
\[
|c_3 - 2 B c_1 c_2 + D c_1^3| \le 2.
\]
Thus, we have
\begin{align*}
	|\gamma_3| \le \frac{1}{24}.
\end{align*}
To establish the sharpness of the inequality, we consider the function $f_{7}$ 
defined in (\ref{te1}) with 
\[
p(z) = \frac{1+z^3}{1-z^3}.
\]
For this choice, the function $f_{7}$ belongs to the class $\mathcal{C}_e$, and its series expansion is given by
\[
f_{7}(z) = z + \frac{1}{12} z^4 + \cdots.
\]
We see that 
\begin{align*}
	|\gamma_3| = \frac{1}{2} \left| a_4 - a_2 a_3 + \frac{1}{3} a_2^3 \right|=\frac{1}{24}.
\end{align*}

\item[(D)]{\bf Sharp bounds of $\gamma_4$}: Using \eqref{Eq-1.3} and \eqref{a22}-\eqref{a25}, it follows that
\begin{align*}
|\gamma_4| &= \frac{1}{2}\bigg|a_5 - a_2 a_4 + a_2^2 a_3 - \frac{1}{2} a_3^2 - \frac{1}{4} a_2^4\bigg| \\
&= \frac{1}{2}\bigg|\left( \frac{1}{5760} c_1^4 - \frac{1}{480} c_1^2 c_2 + \frac{1}{240} c_1 c_3 + \frac{1}{40} c_4 \right)
- \left( \frac{1}{384} c_1^2 c_2 -\frac{1}{4608} c_1^4 + \frac{1}{96} c_1 c_3 \right) \\
&\quad + \left( \frac{1}{768} c_1^4 + \frac{1}{192} c_1^2 c_2 \right)
- \left( \frac{1}{4608} c_1^4 + \frac{1}{576} c_1^2 c_2 + \frac{1}{288} c_2^2 \right)
- \frac{1}{1024} c_1^4\bigg|\\
&=\frac{1}{2} \left| \frac{23}{46080} c_1^4 -  \frac{7}{5760} c_1^2 c_2 -\frac{1}{160} c_1 c_3 - \frac{1}{288} c_2^2 + \frac{1}{40}c_4 \right|\\
&=\frac{1}{80} \left| -\frac{23}{1152} c_1^4 + \frac{7}{144} c_1^2 c_2 + \frac{1}{4} c_1 c_3 + \frac{5}{36} c_2^2 - c_4  \right|\\
&=\frac{1}{16} \left| \gamma  c_1^4 - \frac{3}{2} \beta c_1^2 c_2 +2\alpha c_1 c_3 + \lambda  c_2^2 - c_4 \right|
\end{align*}
where $\gamma =-\frac{1}{192}$, $\lambda  = \frac{1}{4}$, $\beta=\frac{1}{12}$ and $\alpha = \frac{1}{4}$. Observe that
\begin{align*}
	\begin{aligned}
		& 8 \lambda (1 - \lambda)\Big\{ (\alpha \beta - 2 \gamma)^2 + (\alpha (\lambda + \alpha) - \beta)^2 \Big\} 
		+ \alpha (1 - \alpha) (\beta - 2 \lambda \alpha)^2 \\
		&\quad- 4 \alpha^2 (1 - \alpha)^2 \lambda (1 - \lambda) = -\frac{45}{2048} < 0.
	\end{aligned}
\end{align*}

This confirms that all the hypotheses of Lemma \ref{L5} are satisfied. Therefore, we have
\begin{align*}
	\left| \gamma  c_1^4 - \frac{3}{2} \beta c_1^2 c_2 + 2\alpha c_1 c_3 + \lambda c_2^2 - c_4 \right| \le 2,
\end{align*}
which immediately implies that
\begin{align*}
	|\gamma_4| \le \frac{1}{8}.
\end{align*}

\noindent To show that this bound is sharp, we consider the function \(f_{4}\) defined in (\ref{te1}) with
\begin{align*}
	p(z) = \frac{1+z^4}{1-z^4}.
\end{align*}

\noindent It is easy to see that $f_4 \in \mathcal{C}_e$, and its series expansion is given by
\[
f_{4}(z) = z + \frac{1}{4} z^5 + \cdots.
\]
and we see that 
\begin{align*}
	|\gamma_4| = \frac{1}{2}\bigg|a_5 - a_2 a_4 + a_2^2 a_3 - \frac{1}{2} a_3^2 - \frac{1}{4} a_2^4\bigg|=\frac{1}{8}.
\end{align*}
This completes the proof.
\end{proof}

\section{{\bf Second-order Hermitian-Toeplitz determinant of logarithmic coefficients for the classes $\mathcal{S}_e^{\ast}$ and $\mathcal{C}_e$.}}\label{Sec-4}
For two natural numbers $q$ and $n$, the Hermitian-Toeplitz determinant of $q$th order for a function $f \in \mathcal{S}$ is defined as
$T_{q,n}(F_f) := \det[a_{ij}],$ where  $a_{ij} = a_{n+j-i}$  for  $ j \ge i,$ $a_{ij} = \overline{a_{ji}}$  for $j < i,$ $a_1 = 1,$  $\overline{a_i} = a_i, \; 1 \le i \le n.$ In particular, since $T_{2,1}(F_f) = 1 - |a_2|^2$, the second-order Hermitian-Toeplitz determinant involving the logarithmic coefficient is therefore written as
\begin{align*}
	T_{2,1}\Big({F_f}/{\gamma}\Big) = \gamma_1^2 - |\gamma_2|^2.
\end{align*}

In view of  \eqref{Eq-1.3}, we have
\begin{align}
	\label{mt1}
	T_{2,1}\Big({F_f}/{\gamma}\Big) = \frac{1}{16} \Big( -a_2^4 + 4 a_2^2 + 4 a_2^2{\rm Re}(a_3) - 4 |a_3|^2 \Big).
\end{align}
It is natural to raise the following question.
\begin{ques}\label{Q-4.1}
	What can we say about the sharp bounds of $T_{2,1}\Big({F_f}/{\gamma}\Big)$ when $f\in\mathcal{S}_e^{\ast}$ or $f\in\mathcal{C}_e$?
\end{ques}
To affirmatively answer Question \ref{Q-4.1}, this section establishes the sharpness of both bounds for $T_{2,1}\!\left({F_f}/{\gamma}\right)$ by presenting two results: Theorem \ref{Th-4.1} for functions $f\in\mathcal{S}_e^{\ast}$ and Theorem \ref{Th-4.2} for $f\in\mathcal{C}_e$.
\begin{theo}\label{Th-4.1}
Let $f(z) = z + a_2 z^2 + a_3 z^3 + \cdots \in \mathcal{S}_e^{\ast}$. Then
\begin{equation}\label{ht1}
-\frac{1}{16} \leq T_{2,1}\!\left({F_f}/{\gamma}\right) \leq \frac{15}{64}.
\end{equation}
Both inequalities in $\eqref{ht1}$ are sharp.
\end{theo}

\begin{proof}
Since $f \in \mathcal{S}_e^{\ast}$, substituting the values of ${a_2}$ and ${a_3}$ from \eqref{a12}, \eqref{a13} into \eqref{mt1}, we obtain
\begin{equation}\label{pt1}
T_{2,1}\!\left({F_f}/{\gamma}\right)
= \frac{1}{1024} \Big( -c_1^4 + 64 c_1^2 + 8 c_1^2 {\rm Re}(c_2) - 16 |c_2|^2 \Big)
\end{equation}
Applying Lemma \ref{L1} to (\ref{pt1}), we obtain
\begin{equation}\label{pt2}
T_{2,1}\!\left({F_f}/{\gamma}\right)
= \frac{1}{1024}\Big( -c_1^4 + 64 c_1^2 - 4 c_1^2 (4 - c_1^2) {\rm Re}(\xi) - 4 (4 - c_1^2)^2 |\xi|^2 \Big).
\end{equation}
Next, we aim to maximize the right-hand side of (\ref{pt2}).  
Since $-{\rm Re}(\xi) \leq |\xi|$, it follows from (\ref{pt2}) that
\begin{align}\label{pt3}
T_{2,1}\!\left({F_f}/{\gamma}\right)
&\leq \frac{1}{1024}\Big( - c_1^4 + 64 c_1^2 +4 c_1^2 (4-c_1^2) |\xi|
- 4 (4-c_1^2)^2 |\xi|^2 \Big) \notag\\
&= \frac{1}{1024} F(p_1^2,|\xi|).
\end{align}
Setting $p_1^2 = x \in [0,4]$ and $|\xi| = y \in [0,1]$, we can write
\begin{equation}\label{F}
F(x,y) = -x^2 + 64x + 4x(4-x)y - 4(4-x)^2y^2.
\end{equation}
Differentiating (\ref{F}) partially with respect to $x$ and $y$, we obtain
\begin{align*}
	\begin{aligned}
		F_x &= -2x - 8 x y - 8 x y^2 + 64 + 16 y + 32 y^2, \\
		F_y &= 4 (4-x) \big( x - 2 (4-x) y \big).
	\end{aligned}
\end{align*}
Solving the system ${F_x = 0}$ and ${F_y = 0}$, we determine that there is no critical point inside the open domain ${(0,4)\times(0,1)}$.\vspace{1.2mm}

On the boundary of the rectangular region $[0,4]\times[0,1]$, the function $F(x,y)$ takes the following forms:
\begin{align*}
	F(0,y) = -64y^2\; \mbox{and}\;
	F(4,y) = 240 \; \text{for all } y \in [0,1],
\end{align*}
and
\begin{align*}
	F(x,0) = -x^2 + 64x \leq 240\; \mbox{and}\;
	F(x,1) = -9x^2 + 112x - 64 \leq 240 \; \text{for all } x \in [0,4].
\end{align*}
From the above discussion, we obtain that
\begin{align*}
	T_{2,1}\!\left({F_f}/{\gamma}\right)
	\leq \frac{240}{1024} = \frac{15}{64}.
\end{align*}
It can be easily shown that the above inequality is sharp in case of the function $f_1$ defined in (\ref{f1}).

\medskip
Next, we aim to minimize the right-hand side of (\ref{pt2}).  
Since ${\rm Re}(\xi) \leq |\xi|$, it follows from (\ref{pt2}) that
\begin{align}\label{pt4}
T_{2,1}\!\left({F_f}/{\gamma}\right)
&\geq \frac{1}{1024}\Big(  -c_1^4 + 64 c_1^2 - 4 c_1^2 (4 - c_1^2) |\xi| - 4 (4 - c_1^2)^2 |\xi|^2 \Big) \notag\\
&= \frac{1}{1024} G(p_1^2,|\xi|).
\end{align}
Setting $p_1^2 = x \in [0,4]$ and $|\xi| = y \in [0,1]$, we can write
\begin{equation}\label{G}
G(x,y) = -x^2 + 64x - 4x(4-x)y - 4(4-x)^2y^2.
\end{equation}
Now, differentiating (\ref{G}) partially with respect to $x$ and $y$, we obtain
\[
\begin{aligned}
G_x &=-2x + 64 -16y + 8 x y - 8 x y^2 + 32 y^2, \\
G_y &=-4 (4-x) \big( x + 2(4-x) y \big).
\end{aligned}
\]
Solving the system $G_x = 0$ and $G_y = 0$, we find that there is no critical point inside the open domain $(0,4)\times(0,1)$.

On the boundary of the rectangular region $[0,4]\times[0,1]$, the function $G(x,y)$ takes the following forms:
\begin{align*}
	G(0,y) = -64y^2\geq -64\; \mbox{and}\;
	G(4,y) = 260,\; \text{for all } y \in [0,1],
\end{align*}
and
\begin{align*}
	G(x,0) = -x^2 + 64x \geq 0, \mbox{and}\;
	G(x,1) = -x^2 + 80x - 64 \geq -64\; \text{for all } x \in [0,4].
\end{align*}
From the above discussion, we deduce that
\begin{align*}
	T_{2,1}\!\left({F_f}/{\gamma}\right)
	\geq -\frac{64}{1024}
	= -\frac{1}{16}.
\end{align*}
It is not hard to show the above inequality is sharp in case of the function $f_2$ defined in (\ref{f2}). This completes the proof.
\end{proof}
\begin{theo}\label{Th-4.2}
Let $f(z) = z + a_2 z^2 + a_3 z^3 + \cdots \in \mathcal{C}_e$. Then
\begin{equation}\label{tt1}
-\frac{1}{144} \leq T_{2,1}\!\left({F_f}/{\gamma}\right) \leq \frac{15}{256}.
\end{equation}
Both inequalities in \eqref{tt1} are sharp
\end{theo}

\begin{proof}
Since $f \in \mathcal{C}_e$, substituting the coefficients $a_2$ and $a_3$ from (\ref{a22})-(\ref{a23}) into (\ref{mt1}), we obtain
\begin{equation}\label{qt1}
T_{2,1}\!\left({F_f}/{\gamma}\right)
= \frac{1}{36864}\Big(-c_1^4 + 576c_1^2 + 32c_1^2\Re(c_2) - 64|c_2|^2\Big).
\end{equation}
Using Lemma \ref{L1}, expression (\ref{qt1}) can be rewritten as
\begin{equation}\label{qt2}
T_{2,1}\!\left({F_f}/{\gamma}\right)
= \frac{1}{36864}\Big(-9c_1^4 + 576c_1^2 - 24c_1^2(4 - c_1^2)\Re(\xi)
- 16(4 - c_1^2)^2|\xi|^2\Big).
\end{equation}

We first seek the maximum of the right-hand side of (\ref{qt2}).  
Since $-\Re(\xi) \leq |\xi|$, inequality (\ref{qt2}) gives
\begin{align}\label{qt3}
T_{2,1}\!\left({F_f}/{\gamma}\right)
&\leq \frac{1}{36864}\Big(-9c_1^4 + 576c_1^2 + 24c_1^2(4 - c_1^2)|\xi|
- 16(4 - c_1^2)^2|\xi|^2\Big) \notag\\
&:= \frac{1}{36864} \Phi(p_1^2,|\xi|).
\end{align}
If we let $p_1^2 = x \in [0,4]$ and $|\xi| = y \in [0,1]$, then
\begin{equation}\label{Fx}
\Phi(x,y) = -9x^2 + 576x + 24x(4-x)y - 16(4-x)^2y^2.
\end{equation}

Differentiating (\ref{Fx}) partially with respect to $x$ and $y$, we have
\[
\begin{aligned}
\Phi_x &= -x(18 + 48y + 32y^2) + (576 + 96y + 128y^2),\\
\Phi_y &= 8(4 - x)\big(3x - 4(4 - x)y\big).
\end{aligned}
\]
The system $\Phi_x = 0$ and $\Phi_y = 0$ has no solution within the open region $(0,4)\times(0,1)$.

On the boundary of the closed rectangle $[0,4]\times[0,1]$, $\Phi(x,y)$ takes the following forms:
\[
\Phi(0,y) = -256y^2\; \mbox{and}\;
\Phi(4,y) = 2160\; \text{for all } y \in [0,1],
\]
and
\[
\Phi(x,0) = -9x^2 + 576x \leq 2160\; \mbox{and}\;
\Phi(x,1) = -49x^2 + 800x - 256 \leq 2160 \quad \text{for } x \in [0,4].
\]
Hence,
\[
T_{2,1}\!\left({F_f}/{\gamma}\right)
\leq \frac{2160}{36864} = \frac{15}{256}.
\]
The above inequality is sharp in case of the function $f_5$ defined in (\ref{f5}).

We now consider the minimum of (\ref{qt2}).  
Since $\Re(\xi) \leq |\xi|$, from (\ref{qt2}) it follows that
\begin{align}\label{qt4}
T_{2,1}\!\left({F_f}/{\gamma}\right)
&\geq \frac{1}{36864}\Big(-9c_1^4 + 576c_1^2 - 24c_1^2(4 - c_1^2)|\xi|
- 16(4 - c_1^2)^2|\xi|^2\Big) \notag\\
&= \frac{1}{36864} \Psi(p_1^2,|\xi|),
\end{align}
where
\begin{equation}\label{Gx}
\Psi(x,y) = -9x^2 + 576x - 24x(4-x)y - 16(4-x)^2y^2.
\end{equation}

\noindent Differentiating (\ref{Gx}) with respect to $x$ and $y$, we obtain
\[
\begin{aligned}
\Psi_x &= x(-18 + 48y - 32y^2) + 576 - 96y + 128y^2,\\
\Psi_y &= -8(4 - x)\big(3x + 4y(4 - x)\big).
\end{aligned}
\]
The system $\Psi_x = 0$ and $\Psi_y = 0$ has no interior solution in $(0,4)\times(0,1)$.

On the boundary of $[0,4]\times[0,1]$, the values of $\Psi$ are
\begin{align*}
	\Psi(0,y) = -256y^2\; \mbox{and}\;
	\Psi(4,y) = 2160\; \text{for all } y \in [0,1],
\end{align*}
and
\begin{align*}
	\Psi(x,0) = -9x^2 + 576x \geq 0\; \mbox{and}\;
	\Psi(x,1) = -49x^2 + 800x - 256 \geq -256\; \text{for } x \in [0,4].
\end{align*}
Hence, we have
\begin{align*}
	T_{2,1}\!\left({F_f}/{\gamma}\right)
	\geq -\frac{256}{36864} = -\frac{1}{144}.
\end{align*}
The above inequality is sharp in case of the function $f_6$ defined in (\ref{f6}). This completes the proof.
\end{proof}
\section{{\bf Generalized Zalcman conjecture for the Class $\mathcal{S}_e^{\ast}$ and $\mathcal{C}_e$.}}\label{Sec-5}
In $1960$, Zalcman conjectured that if $f\in\mathcal{S}$ and is given by \eqref{eq1}, then $|a_n^2-a_{2n-1}|\leq (n-1)^2$ for $n\geq 2$ with equality only for the Koebe function $k(z)=z/(1-z)^2$, or its rotations, which implies the famous Bieberbach conjecture $|a_n|\leq n$ for $n\geq 2$. For $f\in\mathcal{S}$, Ma \cite{Ma-JMAA-1999} proposed a generalized Zalcman conjecture
\begin{align*}
	|a_na_m-a_{n+m-1}|\leq (n-1)(m-1)
\end{align*}
for $m\geq 2,\; n\geq 2$, which is still an open problem. However, Ma \cite{Ma-JMAA-1999} proved this generalized Zalcman conjecture for the classes $\mathcal{S}^*$ and $\mathcal{S}_{\mathbb{R}}$, where $\mathcal{S}_{\mathbb{R}}$ denotes the class of all functions in $\mathcal{S}$ with real coefficients. In $2017$, Ravichandran and Verma \cite{Ravichandran-JMAA-2019} proved the conjecture for starlike and convex functions of given order, and for the class of functions with bounded turning. In \cite{Allu-Pandey-JMAA-2020}, Allu and Pandey proved the Zalcman conjecture and the generalized Zalcman conjecture for the class $\mathcal{U}$ using extream point theory and also proved the generalized Zalcman conjecture for the class $\mathcal{CR}^+$ for the initial coefficients.\vspace{1.2mm}

In this paper, we prove two results (Theorem \ref{Th-5.1} and Theorem \ref{Th-5.2}) regarding the Generalized Zalcman Conjecture for the initial coefficients of functions belonging to the class $\mathcal{S}_e^{\ast}$ or $\mathcal{C}_e$.  It is worth noting that a bound for a subclass is expected to be smaller (tighter) than the conjectured bound for the entire class $\mathcal{S}$, as the subclass represents a more restricted set of functions.
\begin{theo}\label{Th-5.1}
Let $f(z) = z + a_2 z^2 + a_3 z^3 + \cdots \in \mathcal{S}_e^{\ast}$. Then
\begin{equation}\label{5.1}
|a_2a_3-a_4| \leq \frac{8}{9\sqrt 7}\approx 0.336.
\end{equation}
The inequality \eqref{5.1} is sharp.
\end{theo}
\begin{proof} In view of \eqref{a12}-\eqref{a14}, we obtain
\begin{align}
	\label{z1} |a_2a_3-a_4|&=\bigg|\frac{1}{2}c_1\left(\frac{1}{16}c_1^2+\frac{1}{4}c_2\right)-\left(\frac{1}{24} c_1 c_2 - \frac{1}{288} c_1^3 + \frac{1}{6} c_3\right)\bigg|\nonumber\\
	&=\frac{1}{144} \left|5c_1^3+12c_1c_2-24c_3\right|.
\end{align}

Substituting the value of $c_1, c_2$ and $c_3$ in (\ref{z1}) we have 
\bea\label{z2} |a_2a_3-a_4|&=&\frac{1}{18}\left|5\tau_1^3-6(1-\tau_1^2)\tau_1\tau_2+6(1-\tau_1^2)\tau_1\tau_2^2-6(1-\tau_1^2)(1-|\tau_2|^2)\tau_3\right|\eea
We now divide the proof into the following cases:

\medskip
{\bf Case 1.} Let $\tau_1=1$. Then, from (\ref{z2}) we get
\[|a_2a_3-a_4|=\frac{5}{18}\approx 0.27777.\]

\medskip
{\bf Case 2.} Let $\tau_1=0$. Then, from (\ref{z2}) we have 
\[|a_2a_3-a_4|=\frac{1}{18}|6\tau_3|\leq \frac{1}{3}\approx 0.3333.\]

\medskip
{\bf Case 3.} Let $\tau_1 \in (0,1)$. Applying the triangle inequality in (\ref{z2}) and using the fact that $|\tau_3| \le 1$, we obtain
\bea\label{z3} |a_2a_3-a_4|&\leq &\frac{1}{18}\left(|5\tau_1^3-6(1-\tau_1^2)\tau_1\tau_2+6(1-\tau_1^2)\tau_1\tau_2^2|+|6(1-\tau_1^2)(1-|\tau_2|^2)\tau_3|\right)\nonumber\\
&\leq &\frac{1}{18}\left(|5\tau_1^3-6(1-\tau_1^2)\tau_1\tau_2+6(1-\tau_1^2)\tau_1\tau_2^2|+|6(1-\tau_1^2)(1-|\tau_2|^2)|\right)\nonumber\\
&=&\frac{1}{3}(1-\tau_1^2)\left(\bigg|\frac{5\tau_1^3}{6(1-\tau_1^3)}-\tau_1\tau_2+\tau_1\tau_2^2\bigg|+1-|\tau_2|^2\right)\nonumber\\
&=&\frac{1}{3}(1-\tau_1^2)(|A+B\tau_2+C\tau_2^2|-1-|\tau_2|)\nonumber\\
&=&\frac{1}{3}(1-\tau_1^2)Y(A,B,C), \eea
where $A=\frac{5\tau_1^3}{6(1-\tau_1^2)}$, $B=-\tau_1$ and $C=\tau_1$.\par
We note that $AC > 0$. Hence, we can apply case (i) of Lemma \ref{L2} and discuss the following cases.

A simple computation shows that $2(1-|C|)-|B|=2(1-\tau_1)-\tau_1=2-3\tau_1\leq 0$ when $\tau_1\geq \frac{2}{3}$, otherwise $2(1-|C|)-|B|>0$.\vspace{1.2mm}

\noindent{\bf Case A.} If $1>\tau_1\geq \frac{2}{3}$ then $2(1-|C|)\leq |B|$. Thus from Lemma \ref{L2}, we see that 
\begin{align*}
Y(A,B,C) &= |A| + |B| + |C| \\
&= \frac{5 \tau_1^3}{6(1-\tau_1^3)} + \tau_1 + \tau_1 \\
&= \frac{1}{6(1-\tau_1^3)}\big(12\tau_1 - 7\tau_1^3\big) \\
\end{align*}

In view of the inequality (\ref{z3}), it follows that
\begin{align*}
|a_2a_3-a_4|&\leq \frac{1}{3}(1-\tau_1^2)Y(A,B,C)\\
&=\frac{1}{18} \big(12\tau_1 - 7\tau_1^3\big)\\
&=\frac{1}{18}\psi_1(t)
\end{align*}
where $\psi_1(t) = 12 t - 7 t^3$ for $t \in \left[\frac{2}{3},1\right)$. \vspace{1.2mm}
  
A straightforward calculation shows that $\psi_1'(t) = 12 - 21 t^2$ and $\psi_1''(t) = -42 t < 0.$ 
The critical point is $t_0 = \frac{2}{\sqrt{7}}$. Since $\psi_1''(t) < 0$, $\psi(t)$ attains its maximum at $t_0$, so that $\psi_1(t_0) = \frac{16}{\sqrt{7}}.$ Therefore, we conclude that 
\begin{align*}
	|a_2a_3-a_4|\leq \frac{8}{9\sqrt 7}\approx 0.336.
\end{align*}
\noindent{\bf Case B.} If $0<\tau_1<\frac{2}{3}$, then $2(1-|C|)> |B|$. Thus, from Lemma \ref{L2}, we see that
\begin{align*}
Y(A,B,C) &= 1+|A| + \frac{B^2}{4(1-|C|)} \\
&= \frac{1}{12(1-\tau_1^2)}\big(12-9\tau_1^2 +13\tau_1^3\big)
\end{align*}
In view of the inequality (\ref{z3}), it follows that
\begin{align*}
|a_2a_3 - a_4| &\leq \frac{1}{3}(1-\tau_1^2)Y(A,B,C) \\
&= \frac{1}{36} \big( 12 - 9\tau_1^2 + 13\tau_1^3 \big) \\
&= \frac{1}{36} \psi_2(t),
\end{align*}
where $\psi_2(t) = 12 - 9t^2 + 13t^3$ for $t \in \left(0, \frac{2}{3}\right)$.\vspace{1.2mm}

A simple computation shows that $\psi_2'(t) = t(13t - 6)$.Since $\psi_2'(t) < 0$ for $t \in \left(0, \frac{6}{13}\right)$, the function $\psi_2$ is decreasing on this interval. Conversely, $\psi_2'(t) \geq 0$ for $t \in \left[\frac{6}{13}, \frac{2}{3}\right)$, meaning $\psi_2$ is increasing on $\left[\frac{6}{13}, \frac{2}{3}\right)$. Hence, we have
\begin{align*}
	\max_{t \in \left(0, \frac{2}{3}\right)} \{ \psi_2(t) \} = \max \{ \psi_2(0), \psi_2(\frac{2}{3}) \} = \max \left\{ 12, \frac{320}{27} \right\} = 12.
\end{align*}
Consequently, we have \[|a_2a_3-a_4|\leq \frac{1}{3}\approx 0.333.\]
From the preceding discussion, we conclude that
\begin{align*}
	|a_2a_3-a_4| \leq \frac{8}{9\sqrt 7}\approx 0.336
\end{align*}
which establishes the desired inequality.\vspace{1.2mm}

To show that this inequality is sharp, we consider the function \(f\) defined in (\ref{te1}) with
\begin{align*}
	p(z) = \frac{1+z-z^2-z^3}{1+(1-2t_0)z+(1-2t_0)z^2+z^2}
\end{align*}
where $t_0=\frac{2}{\sqrt{7}}$.
\end{proof}
\begin{theo}\label{Th-5.2}
Let $f(z) = z + a_2 z^2 + a_3 z^3 + \cdots \in \mathcal{C}_e$. Then
\begin{equation}\label{z21}
|a_2a_3-a_4| \leq \frac{1}{12}.
\end{equation}
The inequality \eqref{z21} is sharp.
\end{theo}
\begin{proof}
From the preceding discussion, we have
\begin{align}
\label{z22}
|a_2a_3 - a_4| &= \bigg| \frac{1}{4} c_1 \left( \frac{1}{12} c_2 + \frac{1}{48} c_1^2 \right) - \left( \frac{1}{96} c_1 c_2 - \frac{1}{1152} c_1^3 + \frac{1}{24} c_3 \right) \bigg| \nonumber \\
&= \frac{1}{1152} \left| 7 c_1^3 + 12 c_1 c_2 - 48 c_3 \right|.
\end{align}

Substituting the expressions for $c_1$, $c_2$, and $c_3$ in \eqref{z22} gives
\begin{align}
\label{z23}
|a_2a_3 - a_4| = \frac{1}{144} \big| \tau_1^3 - 18 (1-\tau_1^2) \tau_1 \tau_2 + 12 (1-\tau_1^2) \tau_1 \tau_2^2 - 12 (1-\tau_1^2)(1-|\tau_2|^2) \tau_3 \big|.
\end{align}

To complete the proof, we now consider the following cases:

\medskip
\noindent
\textbf{Case 1.} If $\tau_1 = 1$, then from \eqref{z23} we obtain
\[
|a_2a_3 - a_4| = \frac{1}{144} \approx 0.006944\ldots
\]

\medskip
\noindent
\textbf{Case 2.} If $\tau_1 = 0$, then
\[
|a_2a_3 - a_4| = \frac{1}{144} |12 \tau_3| \le \frac{1}{12} \approx 0.08333\ldots
\]

\medskip
\noindent
\textbf{Case 3.} If $\tau_1 \in (0,1)$, applying the triangle inequality to \eqref{z23} and using $|\tau_3| \le 1$, we have
\begin{align}
\label{z24}
|a_2a_3 - a_4| &\le \frac{1}{144} \Big( \big| \tau_1^3 - 18 (1-\tau_1^2) \tau_1 \tau_2 + 12 (1-\tau_1^2) \tau_1 \tau_2^2 \big| + 12 (1-\tau_1^2)(1-|\tau_2|^2) \Big) \nonumber\\
&= \frac{1}{12} (1-\tau_1^2) \Big( \big| \frac{\tau_1^3}{12(1-\tau_1^2)} - \frac{3}{2} \tau_1 \tau_2 + \tau_1 \tau_2^2 \big| + 1 - |\tau_2|^2 \Big) \nonumber\\
&= \frac{1}{12} (1-\tau_1^2) Y(A,B,C),
\end{align}
where
\[
A = \frac{\tau_1^3}{12(1-\tau_1^2)}, \quad B = -\frac{3}{2} \tau_1, \quad C = \tau_1.
\]

Since $AC > 0$, we can apply case (i) of Lemma \ref{L2}. A direct computation shows that
\[
2(1-|C|)-|B| = 2(1-\tau_1) - \frac{3}{2}\tau_1 = 2 - \frac{7}{2}\tau_1.
\]

Thus, for $\tau_1 \ge \frac{4}{7}$, we have $2(1-|C|) \le |B|$, and Lemma \ref{L2} gives
\begin{align*}
Y(A,B,C) &= |A| + |B| + |C| \\
&= \frac{\tau_1^3}{12(1-\tau_1^2)} + \frac{3}{2}\tau_1 + \tau_1 \\
&= \frac{1}{12(1-\tau_1^2)} \big( 30 \tau_1 - 29 \tau_1^3 \big).
\end{align*}

Consequently, from \eqref{z24}, it is straightforward to show that
\begin{align*}
	|a_2a_3 - a_4| \le \frac{1}{12} (30\tau_1 - 29\tau_1^3) = \frac{1}{12} \psi_3(\tau_1), \quad \tau_1 \in \left[ \frac{4}{7}, 1 \right),
\end{align*}
where $\psi_3(\tau_1) = 30\tau_1 - 29\tau_1^3$. Computing the derivative,
\[
\psi_3'(\tau_1) = 30 - 87 \tau_1^2, \quad \psi_3''(\tau_1) = -174 < 0,
\]
gives the critical point $\tau_1 = \sqrt{10/29} \in [4/7,1)$. Since $\psi_3'' < 0$, $\psi_3$ attains its maximum at $\tau_1 = \sqrt{10/29}$, yielding
\[
|a_2a_3 - a_4| \le \frac{5}{36} \sqrt{\frac{10}{29}} \approx 0.0815.
\]

For $0 < \tau_1 < 4/7$, we have $2(1-|C|) > |B|$. Then Lemma \ref{L2} gives
\begin{align*}
Y(A,B,C) &= 1 + |A| + \frac{B^2}{4(1-|C|)} \\
&= \frac{1}{192(1-\tau_1^2)} \big( 124 \tau_1^3 - 84 \tau_1^2 + 192 \big).
\end{align*}

Thus, we see that
\begin{align*}
	|a_2a_3 - a_4| &\le \frac{1}{12} (1-\tau_1^2) Y(A,B,C)\\& = \frac{1}{2304} \big( 124 \tau_1^3 - 84 \tau_1^2 + 192 \big)\\& = \frac{1}{2304} \psi_4(\tau_1),
\end{align*}
where $\psi_4(\tau_1) := 192 - 84 \tau_1^2 + 124 \tau_1^3$ for $\tau_1 \in \left(0, \frac{4}{7} \right)$.\vspace{1.2mm}  

A simple calculation yields $\psi_4'(\tau_1) = 12 \tau_1(-14 + 31 \tau_1)$.Since $\psi_4'(\tau_1) < 0$ for $\tau_1 \in (0, 14/31)$, the function $\psi_4$ is decreasing on $(0, 14/31)$. Conversely, $\psi_4'(\tau_1) \geq 0$ for $\tau_1 \in [14/31, 4/7)$, meaning $\psi_4$ is increasing on $[14/31, 4/7)$. Hence,
\begin{align*}
	\max_{\tau_1 \in (0,4/7)} \psi_4(\tau_1) = \max\{ \psi_4(0), \psi_4(4/7) \} = \max \{192, 64288/343\} = 192.
\end{align*}

Combining all the above cases, we get the desired inequality
\begin{align*}
	|a_2a_3 - a_4| \le \frac{1}{12} \approx 0.08333.
\end{align*}

To show the bound is sharp, we consider the function $f$ defined in \eqref{te1} with
\begin{align*}
	p(z) = \frac{1+z^3}{1-z^3}
\end{align*}
and $f$ is of the form 
\[f(z)=z+\frac{1}{12}z^4+\cdots.\]
It is easy to see that equality
\begin{equation*}\label{z21}
	|a_2a_3-a_4|=\frac{1}{12}
\end{equation*}
holds for the $f$. This completes the proof.
\end{proof}
\section{{\bf  Generalized Fekete-Szeg\"{o} functional for the classes $\mathcal{S}_e^{\ast}$ and $\mathcal{C}_e$.}}\label{Sec-6}
In 2024, Lecko and Partyka \cite{Lecko2024} investigated the generalized Fekete-Szeg\"{o} functional for the class $\mathcal{S}$ defined by  
\[
F_{\lambda, \mu}(f) := \big| a_3(f) - \lambda a_2(f)^2 \big| - \mu |a_2(f)|,
\]
where $\lambda \in \mathbb{C}$ and $\mu > 0$.  
The coefficients $a_2(f) = a_2$ and $a_3(f) = a_3$ are given by (1).  
Hence, we can write
\bea\label{FG}
F_{\lambda, \mu}(f) = \big| a_3 - \lambda a_2^2 \big| - \mu |a_2|, 
\qquad \lambda \in \mathbb{C}, \ \mu > 0. 
\eea
In this section, our aim is to establish the sharp upper and lower bounds for $F_{\lambda, \mu}(f)$ on the classes $\mathcal{S}_e^{\ast}$ and $\mathcal{C}_e$. The proof relies on the lemma presented below.\vspace{1.2mm}

\begin{theo}
Let $f(z) = z + a_2 z^2 + a_3 z^3 + \cdots \in \mathcal{S}_e^*$. Then
\begin{equation}\label{fg1}
B_1\leq F_{\lambda, \mu}(f)\leq
\begin{cases} 
\frac{1}{4}\big(|3 - 4\lambda| - 4\mu\big), & \text{if } |3 - 4\lambda|\geq 2+4\mu,\\[2mm]
\frac{1}{2}, & \text{if } |3 - 4\lambda|< 2+4\mu.
\end{cases}
\end{equation}
where 
\[
B_1 =
\begin{cases}
-\frac{1}{4}(4\mu-|3 - 4\lambda|), & \text{if } \frac{\mu+1}{2}\geq |3 - 4\lambda|, \\[1.2em]
-\mu \sqrt{\dfrac{2}{|3 - 4\lambda| + 2}}, & \text{if } |3 - 4\lambda|\geq \frac{\mu^2+1}{2}, \\[1.2em]
-\dfrac{|3 - 4\lambda| + 16\mu^2 + 16}{2(|3 - 4\lambda| + 2)}, & \text{if } \frac{\mu+1}{2}<|3 - 4\lambda|< \frac{\mu^2+1}{2}.
\end{cases}
\]
The inequalities in \eqref{fg1} are sharp.
\end{theo}
\begin{proof} Given that $f \in \mathcal{S}_e^{\ast}$, using (\ref{a22}), (\ref{a23}), and (\ref{FG}), one obtains
\begin{align}\label{fg2} F_{\lambda, \mu}(f)&=\big|a_3 - \lambda a_2^2\big| - \mu|a_2|\nonumber\\
&= \left|\frac{1}{4}c_2 + \left(\frac{1}{16} - \frac{\lambda}{4}\right)c_1^2\right|
- \frac{\mu}{2}|c_1|\nonumber\\
&= \frac{1}{16} (|4c_2+(1-4\lambda)c_1|-|8\mu c_1|\nonumber\\
&=\frac{1}{16}\Phi(p_1,p_2)
\end{align}
where $\Phi(p_1,p_2)=|Kp_1^2+Lp_2|-|Jp_1|$, with $K=(1-4\lambda)$, $L=4$,  $J=8\mu$ and $M=|4K+2L|=4|3-4\lambda|$.\par
Since 
\[|2K + L| - |L| - J = 2\big(|3 - 4\lambda| - 2 - 4\mu\big)
\]
The condition $|3 - 4\lambda|\geq 2+4\mu$ implies $|2K + L| \geq |L| +J$. Thus, Lemma \ref{Ls} yields
\[\Phi(p_1,p_2)\leq |4K + 2L| - 2J=4\big(|3 - 4\lambda| - 4\mu\big).\]
Thus, from $(\ref{fg2})$, we obtain the desired inequality
\[ F_{\lambda, \mu}(f)\leq \frac{1}{4}\big(|3 - 4\lambda| - 4\mu\big).\]
The inequality is sharp for the function $f_1$ defined in \eqref{f1}.\vspace{1.2mm}

Similarly, for $|3 - 4\lambda|< 2+4\mu$, we have $|2K + L| < |L| +J$. Now, Lemma \ref{Ls} give
\[\Phi(p_1,p_2)\leq 2|L|=8.\]
 In light of $(\ref{fg2})$, it is easy to see that
\[ F_{\lambda, \mu}(f)\leq \frac{1}{2}.\]
The inequality is sharp for the function $f_2$ defined in \eqref{f2}.\vspace{1.2mm}

Next, we find the lower bound for $F_{\lambda, \mu}(f)$. Let$$J - M - 2|L| = 8\mu - 4|3 - 4\lambda| - 8:=g_2(\lambda).$$The inequality $g_2(\lambda)\geq 0$ is equivalent to$$8\mu - 4|3 - 4\lambda| - 8\geq 0,$$which holds when $\frac{\mu+1}{2}\geq |3 - 4\lambda|$.\vspace{1.2mm}

For $\frac{\mu+1}{2}\geq |3 - 4\lambda|$, we have $J \geq M +2|L|$. Thus, from Lemma \ref{Ls}, we have$$-\Phi(p_1,p_2)\leq 16\mu - 4|3 - 4\lambda|.$$Therefore, from $(\ref{fg2})$, we obtain$$F_{\lambda, \mu}(f) \geq -\frac{1}{4}(4\mu-|3 - 4\lambda|).$$To show that the inequality is sharp, consider the function $f_1$ defined in $\eqref{f1}$.

Moreover, we see that
\[ J^2 - 2|L|(M + 2|L|) = 64\mu^2 - 32|3 - 4\lambda| - 64\leq 0.\]
This yields $|3 - 4\lambda|\geq \frac{\mu^2+1}{2}$, from which Lemma \ref{Ls} provides
\[-\Phi(p_1,p_2)\leq 16\mu \sqrt{\dfrac{2}{|3 - 4\lambda| + 2}}.\]
Thus, it follows from $(\ref{fg2})$ that
\[F_{\lambda, \mu}(f) \geq -\mu \sqrt{\dfrac{2}{|3 - 4\lambda| + 2}}.\]

To show sharpness of the inequality, consider the function $f_7$ defined by (\ref{te1}) with
\[
p(z) = \frac{1 + (t_1 t_2 + t_1)z + t_2 z^2}{1 + (t_1 t_2 - t_1)z - t_2 z^2},
\]
where
\[
t_1 = \sqrt{\frac{2|L|}{M + 2|L|}}, \qquad t_2 = -\frac{|L|(4K + 2L)}{L|4K + 2L|}.
\]
Then, $q_1 = 2t_1$ and $ q_2 = 2t_2^2 + 2(1 - t_1^2)t_2,$
and it gives
\[
|Kq_1^2 + Lq_2| = \frac{|L|(4K + 2L) - |L|(4K + 2L)}{|2K + L| + L} = 0.
\]
Thus,
\[
\Phi(q_1, q_2) = |Kq_1^2 + Lq_2| - |Jq_1| = -2Jt_1 = -2J \sqrt{\frac{2|L|}{M + 2|L|}}.
\]
Therefore from (\ref{fg2}) we have 
\begin{align*}
	F_{\lambda, \mu}(f) = -\mu \sqrt{\dfrac{2}{|3 - 4\lambda| + 2}}
	\end{align*}
Finally, in the range $\frac{\mu+1}{2}<|3 - 4\lambda|< \frac{\mu^2+1}{2}$, Lemma \ref{Ls} implies that
\[ -\Phi(p_1,p_2)\leq 2|L| + \dfrac{J^2}{M + 2|L|} 
= \dfrac{8|3 - 4\lambda| + 16\mu^2 + 16}{|3 - 4\lambda| + 2},\]
\begin{align*}
	F_{\lambda, \mu}(f) \geq -\dfrac{|3 - 4\lambda| + 16\mu^2 + 16}{2(|3 - 4\lambda| + 2)}.
\end{align*}
To show sharpness of above inequality, we consider the function $f_8$ defined by (\ref{te1}) with
\[
p(z) = \frac{1 + (t_1 t_2 + t_1)z + t_2 z^2}{1 + (t_1 t_2 - t_1)z - t_2 z^2},
\]
where
\[
t_1 = \frac{J}{M + 2|L|}\; \mbox{and}\; t_2 = -\frac{|L|(4K + 2L)}{L|4K + 2L|}.
\]
Then, $q_1 = 2t_1$ and $ q_2 = 2t_2^2 + 2(1 - t_1^2)t_2,$
and we have
\begin{align*}
|Kq_1^2 + Lq_2|^2 &= |(4K + 2L)t_1^2 + 2L(1 - t_1^2)t_2|^2\\
&= M^2 t_1^4 + 4\,\mathrm{Re}\big(L\overline{(4K + 2L)}t_1^2(1 - t_1^2)t_2\big)+ 4L^2(1 - t_1^2)^2\\
&= M^2 t_1^4 - 4t_1^2(1 - t_1^2)M|L| + 4L^2(1 - t_1^2)^2\\
&= \left(Mt_1^2 - 2|L|(1 - t_1^2)\right)^2\\
&= \left(\frac{J^2}{M + 2|L|} - 2|L|\right)^2.
\end{align*}

Since $J^2 > 2|L|(M + 2|L|)$, we get
\[
|Kq_1^2 + Lq_2| = \frac{J^2}{M + 2|L|} - 2|L|,
\]
which implies
\begin{align}\label{hhh1}
\Phi(q_1, q_2) = |Kq_1^2 + Lq_2| - |Jq_1| 
= -\left( 2|L| + \frac{J^2}{M + 2|L|} \right).
\end{align}

Thus, from (\ref{fg2}) and  (\ref{hhh1}), we have 
\[F_{\lambda, \mu}(f) = -\dfrac{|3 - 4\lambda| + 16\mu^2 + 16}{2(|3 - 4\lambda| + 2)}.\]

This completes the proof.
\end{proof}
\begin{theo}
Let $f(z) = z + a_2 z^2 + a_3 z^3 + \cdots \in \mathcal{C}_e$. Then
\begin{equation}\label{gf1}
B_1\leq F_{\lambda, \mu}(f)\leq
\begin{cases} 
\frac{1}{4}(|1 - \lambda| - 2\mu), & \text{if } |1 - \lambda| \geq \frac{2}{3}(2 + 3\mu),\\[2mm]
\frac{1}{6}, & \text{if } |1 - \lambda| < \frac{2}{3}(2 + 3\mu).
\end{cases}
\end{equation}
where 
\[
B_2 =
\begin{cases}
-\dfrac{2\mu - |1 - \lambda|}{4}, & \text{if } \dfrac{3\mu - 2}{3} \ge |1 - \lambda|, \\[1.2em]
-\dfrac{1}{2}\mu \sqrt{\dfrac{2}{3|1 - \lambda| + 2}}, & \text{if } \dfrac{9\mu^2 - 4}{6} \le |1 - \lambda|, \\[1.2em]
-\dfrac{9\mu^2 + 6|1 - \lambda| + 4}{12(3|1 - \lambda| + 2)}, & \text{if } \dfrac{9\mu^2 - 4}{6} > |1 - \lambda| > \dfrac{3\mu - 2}{3}.
\end{cases}
\]
The inequalities in \eqref{gf1} are sharp.
\end{theo}
\begin{proof}
Since $f \in \mathcal{C}_e$, in view of (\ref{a12}), (\ref{a13}), and (\ref{FG}), we obtain
\begin{align}\label{gfe1}
F_{\lambda,\mu}(f)
&= \left| \frac{1}{12} c_2
+ \left( \frac{1}{48} - \frac{\lambda}{16} \right) c_1^2 \right|
- \frac{\mu}{4} |c_1| \nonumber\\
&= \frac{1}{48} \big|(1-3\lambda)c_1^2 + 4 c_2\big| - |12 \mu c_1| \nonumber\\
&= \frac{1}{48} \, \Phi(p_1, p_2),
\end{align}
where $K = 1 - 3\lambda$, $L = 4$, and $J = 12 \mu$.  
Also, we easily conclude that 
\[
M = |4K + 2L| = 12 |1 - \lambda|.
\]

For the lower bound, we have
\[
|2K + L| - (|L| + J) = 6|1 - \lambda| - (4 + 12\mu).
\]

If $|1 - \lambda| \geq \frac{2}{3}(2 + 3\mu)$, then $|2K + L| \geq |L| + J$. Hence, from Lemma \ref{Ls}, we deduce that
\[
\Phi(p_1, p_2) \leq |4K + 2L| - 2J = 12|1 - \lambda| - 24\mu.
\]
Thus, it follows from $(\ref{gfe1})$ that
\[F_{\lambda,\mu}(f) \leq \frac{1}{4}(|1 - \lambda| - 2\mu).\]
The inequality is sharp for the function $f_6$ defined in (\ref{f6}).\vspace{1.2mm}

If $|1 - \lambda| < \frac{2}{3}(2 + 3\mu)$, then $|2K + L| < |L| + J$. Hence, from Lemma \ref{Ls}, we deduce that
\[
\Phi(p_1, p_2) \leq 2|L| = 8.
\]
Hence, from $(\ref{gfe1})$, we obtain the desired inequality

\[F_{\lambda,\mu}(f) \leq \frac{1}{6}.\]

The inequality is sharp for the function $f_5$ defined in (\ref{f5}). Note that 
\[
\begin{aligned}
J - (M + 2|L|) &= 12\mu - 12|1 - \lambda| - 8,\\[2mm]
J^2 - 2|L|(M + 2|L|) &= 144\mu^2 - 96|1 - \lambda| - 64.
\end{aligned}
\]
If $\frac{3\mu-2}{3}\geq |1 - \lambda|$, then $J \geq M + 2|L|$. Hence, by applying Lemma \ref{Ls}, we obtain that
\[-\Phi(p_1,p_2)\leq 24\mu-12|1-\lambda|.\]
Therefore, from $(\ref{gfe1})$, we have 
\[F_{\lambda,\mu}(f)\geq -\dfrac{2\mu-|1-\lambda|}{4}.\]

The inequality is sharp for the function $f_6$ defined in (\ref{f6}).

If $\frac{9\mu^2-4}{6}\leq |1 - \lambda|$, then the conditions $J^2 \leq 2|L|(M + 2|L|)$ and $J \not\geq M + 2|L|$ hold. Therefore, by Lemma \ref{Ls}, we deduce that
\[-\Phi(p_1,p_2)\leq 24\mu \sqrt{\frac{2}{3|1 - \lambda| + 2}}.\]

Thus from (\ref{gfe1}) we have 
\[F_{\lambda,\mu}(f)\geq -\dfrac{1}{2}\mu \sqrt{\dfrac{2}{3|1 - \lambda| + 2}}.\]

To show the sharpness, we consider the function $f_9$ defined by (\ref{e1}) with
\[
p(z) = \frac{1 + (t_1 t_2 + t_1)z + t_2 z^2}{1 + (t_1 t_2 - t_1)z - t_2 z^2},
\]
where
\[
t_1 = \sqrt{\frac{2|L|}{M + 2|L|}}\; \mbox{and}\; t_2 = -\frac{|L|(4K + 2L)}{L|4K + 2L|}.
\]
Then, $q_1 = 2t_1$ and $ q_2 = 2t_2^2 + 2(1 - t_1^2)t_2,$
and it gives
\[
|Kq_1^2 + Lq_2| = \frac{|L|(4K + 2L) - |L|(4K + 2L)}{|2K + L| + L} = 0.
\]
Thus,
\[
\Phi(q_1, q_2) = |Kq_1^2 + Lq_2| - |Jq_1| = -2Jt_1 = -2J \frac{2|L|}{M + 2|L|}.
\]
Therefore, from $(\ref{gfe1})$, we have 
\[F_{\lambda,\mu}(f)= -\dfrac{1}{2}\mu \sqrt{\dfrac{2}{3|1 - \lambda| + 2}}.\]

Furthermore, when $\frac{9\mu^2-4}{6} > |1 - \lambda| > \frac{3\mu-2}{3}$, since $J^2 \not\leq 2|L|(M + 2|L|)$ and $J \not\geq M + 2|L|$ hold, application of Lemma \ref{Ls} yields that
\begin{align*}
	 -\Phi(p_1,p_2) \leq \frac{36\mu^2 + 24|1 - \lambda| + 16}{3|1 - \lambda| + 2}.
\end{align*}

Hence, from $(\ref{gfe1})$, we obtain

\[ F_{\lambda,\mu}(f)\geq -\dfrac{9\mu^2 + 6|1 - \lambda| + 4}{12(3|1 - \lambda| + 2)}.\]

To show sharpness of above inequality, we consider the function $f_{10}$ defined by (\ref{e1}) with
\[
p(z) = \frac{1 + (t_1 t_2 + t_1)z + t_2 z^2}{1 + (t_1 t_2 - t_1)z - t_2 z^2},
\]
where
\[
t_1 = \frac{J}{M + 2|L|}, \qquad t_2 = -\frac{|L|(4K + 2L)}{L|4K + 2L|}.
\]
Then, $q_1 = 2t_1$ and $ q_2 = 2t_2^2 + 2(1 - t_1^2)t_2,$ and we have
\begin{align*}
|Kq_1^2 + Lq_2|^2 &= |(4K + 2L)t_1^2 + 2L(1 - t_1^2)t_2|^2\\
&= M^2 t_1^4 + 4\,\mathrm{Re}\big(L\overline{(4K + 2L)}t_1^2(1 - t_1^2)t_2\big)+ 4L^2(1 - t_1^2)^2\\
&= M^2 t_1^4 - 4t_1^2(1 - t_1^2)M|L| + 4L^2(1 - t_1^2)^2\\
&= \left(Mt_1^2 - 2|L|(1 - t_1^2)\right)^2\\
&= \left(\frac{J^2}{M + 2|L|} - 2|L|\right)^2.
\end{align*}

Because $J^2 > 2|L|(M + 2|L|)$, it follows that
\[
|Kq_1^2 + Lq_2| = \frac{J^2}{M + 2|L|} - 2|L|,
\]
which implies
\begin{align}\label{hhh2}
\Phi(q_1, q_2) = |Kq_1^2 + Lq_2| - |Jq_1| 
= -\left( 2|L| + \frac{J^2}{M + 2|L|} \right).
\end{align}

Thus, from (\ref{gfe1}) and  (\ref{hhh2}), we have 
\[F_{\lambda, \mu}(f) =-\dfrac{9\mu^2 + 6|1 - \lambda| + 4}{12(3|1 - \lambda| + 2)}.\]

In view of the above discussion, we conclude that
\begin{align*}
	F_{\lambda,\mu}(f)\geq 
	\begin{cases} -\dfrac{2\mu-|1-\lambda|}{4},\quad\quad\quad\quad\;\;\text{if}\;\dfrac{3\mu-2}{3}\geq |1 - \lambda|\vspace{2mm}\\
		-\dfrac{1}{2}\mu \sqrt{\dfrac{2}{3|1 - \lambda| + 2}},\quad\quad\;\;\text{if}\;\dfrac{9\mu^2-4}{6}\leq |1 - \lambda|\vspace{2mm}\\
		-\dfrac{9\mu^2 + 6|1 - \lambda| + 4}{12(3|1 - \lambda| + 2)},\quad\quad \;\text{if}\;\dfrac{9\mu^2-4}{6}> |1 - \lambda|>|\dfrac{3\mu-2}{3}.
	\end{cases}
\end{align*}
This completes the proof.
\end{proof}

\vspace{0.2in}

\noindent \textbf {Conflict of interest:} The author declares that there is no conflict of interest regarding the publication of this paper.\vspace{1.2mm}

\noindent{\bf Funding:} There is no funding received from any organizations for this research work.\vspace{1.2mm}

\noindent \textbf {Data availability statement:}  Data sharing is not applicable to this article as no database were generated or analyzed during the current study.
\vspace{1.2mm}

\noindent  {\bf Author contributions.} All authors contributed equally to the conceptualization, investigation, and writing of this manuscript. 
All authors have read and approved the final version of the paper.

\end{document}